\documentclass[12pt]{amsart}
\usepackage{geometry} 
\usepackage{amssymb}
\usepackage{wasysym} 
\usepackage[utf8]{inputenc} 
\usepackage{textcomp} 
\usepackage{color}
\usepackage{tikz}
\usepackage{graphicx}  
\usepackage{amscd}
\usepackage{flafter}  
\usepackage{bm}  
\usepackage{todonotes}

\usepackage{pdfsync}  

\DeclareMathOperator{\SE}{SE}

\DeclareMathOperator{\Emb}{Emb}
\DeclareMathOperator{\Diff}{Diff}

\DeclareMathOperator{\vol}{vol}

\DeclareMathOperator{\graph}{graph}
\DeclareMathOperator{\discrete}{discrete}

\newcommand{\bulk}{\ensuremath{\aquarius}}
\newcommand{\Flow}{\ensuremath{\mathrm{Flow}}}
\newcommand{\flow}{\ensuremath{\mathrm{flow}}}
\newcommand{\lb}{\ensuremath{\langle}}
\newcommand{\rb}{\ensuremath{\rangle}}
\newcommand{\pder}[2]{\ensuremath{\frac{\partial #1}{\partial #2}}}
\newcommand{\body}{\ensuremath{\mathcal{B}}}

\newtheorem{prop}{Proposition}
\newtheorem{corr}{Corrollary}

\begin{document}
\title{Swimming as a limit cycle}

\author{Henry O. Jacobs}

\maketitle

\begin{abstract}
  Steady swimming can be characterized as both periodic and stable.  These characteristics are the very definition of limit cycles, and so we ask ``Can we view swimming as a limit cycle?''  In this paper we will find that the answer is ``yes''.  We will define a class of dissipative systems which correspond to the passive dynamics of a body immersed in a Navier-Stokes fluid (i.e. the dynamics of a dead fish).  Upon performing reduction by symmetry we will find a hyperbolically stable fixed point which corresponds to the stability of a dead fish in stagnant water.  Given a periodic force on the shape of the body we will invoke the persistence theorem to assert the existence of a loop which approximately satisfies the exact equations of motion.  If we lift this loop with a phase reconstruction formula we will find that the lifted loops are not loops, but stable trajectories which represent regular periodic motion reminiscent of swimming.
\end{abstract}
\tableofcontents

\pagebreak
\section{Introduction}
  The motion of steady swimming appears periodic \cite{Shadwick1998}.  Thus the question:
  \vspace{2ex}
  \begin{center}
  	\it Can we reasonably interpret steady swimming as a limit cycle?
   \end{center}
   \vspace{2ex}
  In this paper we will assert that the answer to this question is ``yes''.  But why would we want to interpret swimming as a limit cycle?  Well firstly, a limit-cycle interpretation builds upon an existing body of knowledge derived from experimental and computational observations \cite{Alben_Shelley_2005}, \cite{Liao2003a}.  Moreover, this interpretation could be of interest to control engineers because robust and regular behavior reduces the complexity of controllers.

\begin{figure}[h]
\centering
	\includegraphics[ width = 0.7\textwidth]{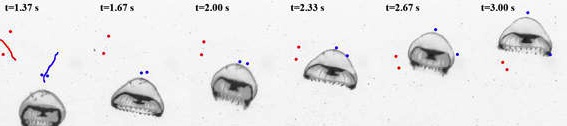}
	\caption{The periodic motion of a jellyfish.  Photo taken from \cite{Katija2011} courtesy of Kakani Katija Young.} \label{fig:jellyfish}
\end{figure}

\paragraph{\bf Main Contributions}
  We will understand the system consisting of a body immersed in a fluid as a dissipative system evolving on a space $\mathcal{A}$.  One observes the system is invariant with respect to the group of rigid rotations and translations, $\SE(d)$.  This observation suggests that one can describe the system evolving on the quotient manifold $\frac{\mathcal{A}}{\SE(d)}$.  The main contributions of this paper are:
\begin{itemize}
	\item We prove the existence of a hyperbolic stable point in the quotient space $\frac{\mathcal{A}}{\SE(d)}$.
	\item We then add a periodic potential energy which models the periodic contraction of muscles.  If $\mathcal{A}$ were finite dimensional, then adding this periodic perturbation would transform a hyperbolic stable point into a stable limit cycle.  However, $\mathcal{A}$ is not finite dimensional.  Nonetheless, we will prove the existence of loops in $\frac{\mathcal{A}}{\SE(d)}$ which approximately satisfy the equations of motion to arbitrarily good accuracy.
	\item We prove that the dynamics in $\mathcal{A}$ which corresponds to loops in $\frac{\mathcal{A}}{\SE(d)}$ are given by regular periodic motions, where each period is related to the previous by a rigid rotation and translation.
\end{itemize}
   
   This all suggests that orderly behavior such as steady swimming in periodically perturbed systems could be \emph{the rule and not the exception}.
   
   In greater detail, the paper proceeds as follows. We will describe the motion of a flexible body immersed in a Navier-Stokes fluid as a dissipative system on a tangent bundle $TQ$.  This passive, or unactuated system, will be referred to as the ``dead fish'' system to specialize the terminology to the situation at hand.  We will then find this systems exhibits two symmetries.  The first symmetry corresponds to the invariance with respect to the particle-labeling of the fluid and is represented by a Lie group $G$.  Reduction by $G$ will send us from dynamics on $TQ$ to dynamics on $(TQ) / G$.  The second symmetry is invariance with respect to global rotations and translations of the system, and will send us from dynamics on $(TQ)/G$ to $\frac{(TQ) / G}{\SE(d)}$.  On this final space we will find a hyperbolic stable point.  We will then add a periodic force which acts on the shape of the body to model periodic muscle contractions of a fish.  Then we will attempt to use the persistence theorem to obtain a stable limit cycle in this periodically perturbed system.  I say ``attempt'' because we will fail.  To get past this failure we will be aim for a more modest goal and successfully use persistence theorem to assert the existence of a loop which approximately satisfies the exact equations of motion (to an arbitrarily good accuracy).  Upon lifting a given loop in $\frac{(TQ) / G}{\SE(d)}$ to $(TQ)/G$ we will find regular motion whereupon each period of the perturbation is a constant rigid rotation and translation of the previous period.  Such behavior is reminiscent of steady swimming.  Schematically, this idea is represented via the commutative diagram

\begin{align}
\begin{CD}
\text{dead fish} @> \text{periodic pertubation} >>  \text{swimming} \\
@V{\pi_{\SE(3)\backslash}}VV @VV{\pi_{\SE(3)\backslash }}V\\
\frac{ \text{dead fish} }{\SE(3)} @> \text{periodic perturbation} >> \frac{\text{perturbed dead fish}}{\SE(3)}.
\end{CD} \label{cd:1}
\end{align}

The top left corner of diagram \eqref{cd:1} represents a dead fish in water as a dissipative mechanical system.  If we reduce the system by $\SE(3)$ we will get a system on a quotient space symbolized by the bottom left corner of diagram \eqref{cd:1}.  In this bottom left corner we find a hyperbolic stable point which asserts ``a motionless body in stagnant water is a stable state for dead fish''.  The horizontal arrows of diagram \eqref{cd:1} correspond to the addition of a time periodic pertubation on the shape of the body.  This perturbation produces a stable limit cycle in the bottom right corner of diagram \eqref{cd:1} via the persistence theorem\footnote{There are analytic concerns however, and we deal with them in \S\ref{sec:asy}}.  Finally, the existence of this periodic orbit in the bottom right corner of diagram \eqref{cd:1} implies the existence of a family of curves which one could reasonably refer to as ``swimming'' in the top right corner of \eqref{cd:1}.

\subsection{ Background Material}
 There presently exists a substantial body of knowledge in the form of computational and biological experiments which are consistent with and support the hypothesis that swimming could be interpreted as a limit cycle. For example,  experiments involving tethered unactuated (i.e. dead) fish immersed in a flow behind a bluff body suggest an ability to passively harvest energy from the surrounding vorticity of the flow.  The same studies also provide a relevant example of oscillatory behavior as a stable state for an unactuated system (e.g. that of a dead fish) \cite{Liao2003a}, \cite{Liao2003b}.
 
   For the case of live fish, periodic motor neuron actuation has been recorded directly and periodic internal elastic forces have been approximated via linear elasticity models \cite{Shadwick1998}.  Additionally, numerical experiments involving rigid bodies with oscillating forces suggest that uniform motion (i.e. flapping flight) is an attracting state for certain pairs of frequencies and Reynolds numbers \cite{Alben_Shelley_2005}.  Finally, experiments involving flexible paper in vertical oscillating flows suggest that the vortices shed have the apparent result of stabilizing a top-heavy body despite steady state analysis which would suggest instability \cite{LiuRistroph2012}.  This body of knowledge suggests further investigation into the role of non-stationary flows in steady swimming via undulatory motion.  In particular, we will be investigating non-stationary flows which result from periodic perturbations of muscle fibers.
 
   In order for us to study non-stationary flows we must have a language which is efficient so that our simple idea is expressed with equal simplicity.  In addition, this language must be rich enough to express fluid-solid systems.  We find that the language of gauge theory is a particular good choice.  Specifically, we will invoke the standard bundle structure of guage theory by setting the gauge symmetry equal to the particle relabeling symmetry of the fluid.  This was first investigated in detail for the case of propulsion in Stokes fluids \cite{ShapereWilczek1989} and potential flow in ideal fluids \cite{Kelly1998} and \cite{Kelly2000}.  In these gauge theories, the shape of the body was controlled explicitly and the locomotion which resulted was seen as a consequence of the curvature of a physically meaningful principal connection.  The framework of \cite{Kelly1998} for arbitrary bodies was specialized to articulated bodies in \cite{Kanso2005} and used for motion planning in \cite{Melli2006}.  More recently,\cite{Vankerschaver2009} was able to extend \cite{Kelly2000} to flow with point vortices with the use of Lagrange-Poincar\'{e} reduction as performed in \cite{CeMaRa2001}.  Finally, recent work has expanded this picture to Euler and Navier-Stokes fluids with arbitrary circulation through the use of Lie groupoids \cite{JaVa2012} in analogy with the classic work of V I Arnold on Lie groups \cite{Arnold1966}.

\subsection{Organization and Motivation of our approach}
  The objective of this paper is to demonstrate that swimming can be interpreted as a limit cycle.  Specifically, we will show that this interpretation may be viewed as a periodic pertubation of the stable state corresponding to a dead fish in stagnant water.  Therefore, we will begin this paper with a discussion in \S\ref{sec:limit_cycles} on limit cycles and demonstrate how a periodically perturbed hyperbolic stable points yields a limit cycle.  Because the first step in this journey is the identification of a hyperbolic stable point, we will seek such a stable point in the realm of fluid-solid systems.  This search for stable points begins with the Lagrangian mechanical description of fluids (\S \ref{sec:passive fluids}) and solids (\S \ref{sec:passive solids}) and finally fluid-solid interaction (\S \ref{sec:passive fsi}).  However, despite having the Lagrangian formalism specified we will not find a hyperbolic stable  point until we have made further preparations.  Our system is too large and there are a number of redundancies which must be removed.  The two redundancies which will be ``quotiented away'' are the particle labels of the fluid  (\S \ref{sec:passive fsi}) and the frame of reference (\S \ref{sec:SE}).   Finally, our quest for a hyperbolic stable point comes to an end in section \S \ref{sec:asy} where we identify one in the quotient space.  With this fixed point we will be able to use the ``persistence theorem'' to assert the existence of loops which approximate the exact dynamics on the quotient space.  Finally, in \S \ref{sec:swimming} we will apply a phase reconstruction formula lift these loops to unquotiented spaces.  The lifted loops will no longer be loops, but stable trajectories which represent regular periodic motion, generally reminiscent of swimming.

\subsection{Conventions and Notation}
  All objects and morphisms will be assumed to be sufficiently smooth.  Moreover, we will not address the existence or uniqueness of solutions for fluid structure systems.  In particular we will assume that solutions exist for all time.  If $M$ is a smooth manifold then we will denote the tangent bundle by $\tau_M :TM \to M$.  The set of vector fields on $M$ will be denoted by $\mathfrak{X}(M)$.  More generally, given a fiber-bunde $\pi:E \to M$, we denote the set of sections of $E$ by $\Gamma(E)$.  We will denote the unit circle by $S^1$ and the Lie group of rigid rotations and translations of $
\mathbb{R}^d$ by $\SE(d)$.

\subsection{Acknowledgements}
The italicized question in the second line of the introduction was posed to me by my wife, Erica J. Kim, while she was studying humming birds.  Additionally, Sam Burden, Ram Vasudevan, and Humberto Gonzales provided much insight into how to frame this work for engineers.  I would also like to thank Professor Shankar Sastry for allowing me to stay in his lab for a year and meet these people who are outside of my normal research circle.  At the same time, the original version of this paper was written in the context of Lie groupoid theory, where the guidance of Alan Weinstein was invaluable.  Much of the groupoid theory is hidden in the present publication, and I hope to reveal the groupoid theory explicitly at a later date.  Jaap Eldering and Joris Vankerschaver have given me more patience than I may deserve by reading my papers and checking my claims.  The final draft of this paper was solidified with the help of Darryl Holm by multiple long meetings going through the paper one line at a time.  This research has been supported by the European Research Council Advanced Grant 267382 FCCA and NSF grant CCF-1011944.

\section{Limit Cycles} \label{sec:limit_cycles}
  Let $M$ be a finite dimensional Banach manifold and let $X \in \mathfrak{X}(M)$.  Moreover, assume we have found a hyperbolically stable fixed point $x_0 \in M$.  Then we may embed this autonomous dynamical system in the time periodic augmented phase space $S^1 \times M$ with the coordinates $(t,x)$ by using the vector field $(\partial_t ,X) \in \mathfrak{X}( S^1 \times M)$.  If we do this then the loop
  \begin{align}
  	\Gamma_0 = \{ ( t, x_0 ) : t \in S^1 \} \subset S^1 \times M \label{eq:cycle}
  \end{align}
  is a hyperbolically stable limit cycle of $(\partial_t,X)$ (see figure \ref{fig:limit_cycle} B).
  \begin{figure}[h]
     \centering
     \includegraphics[width=4in]{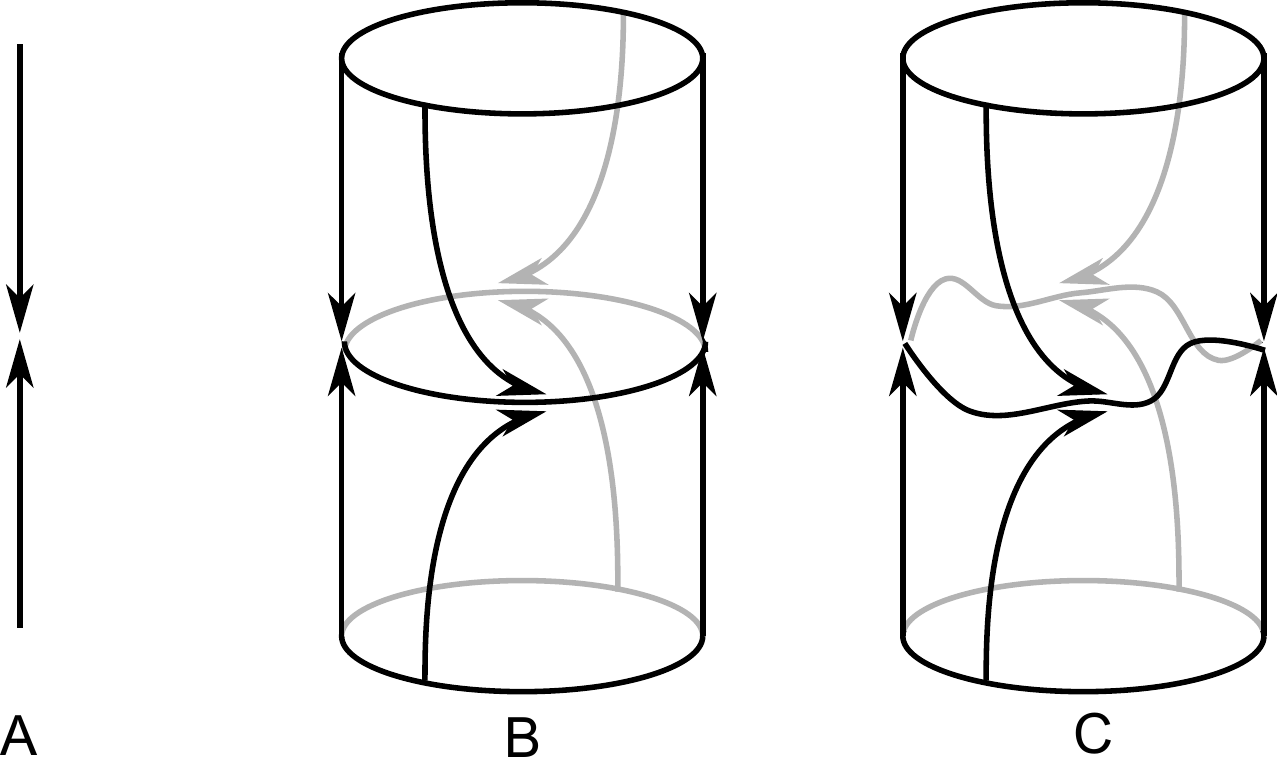} 
     \caption{ (A) Depicts a one dimensional system with a hyperbolic stable point.  The embedding of the system in (A) into time-periodic augmented phase space exhibits a limit cycle as shown in (B).  Finally, a limit cycle will persist under a time-periodic perturbation as a diffeomorphism of the original limit cycle.  This is drawn in (C). }
     \label{fig:limit_cycle}
  \end{figure}
  In fact, $\Gamma_0$ is a special case of a normally hyperbolically invariant submanifold.  Therefore under sufficiently small perturbations $(\partial_t , X) \mapsto  (\partial_t,X+ \delta\!X)$ there will still exist a hyperbolically stable limit cycle, $\Gamma_{\epsilon}$, in the perturbed system which is close to $\Gamma_0$.  This is a consequence of the \emph{persistence theorem} \cite{Fenichel1971,Hirsch77}.

  It is notable that the first step in going down this path is to find a hyperbolic stable point.  We will do this for the case of a dead fish as follows.  In the next section we will derive the equations of motion for a dead fish as an instance of the Lagrange-d'Alembert equations on a tangent bundle $TQ$.  We will then find two symmetries of the system.  One corresponding to a Lie group $G$ and another corresponding to the Lie group $\SE(d)$.  Upon performing reduction by symmetry we will obtain dynamics on the quotient space
  \[
  	[TQ/G] : = \frac{TQ/G}{\SE(d)}
  \]
  we will find that the vector field on $[TQ/G]$ has a (weak) hyperbolic stable point, $x_0 \in [TQ/G]$.  We could then construct a (weak) hyperbolic stable limit cycle given by equation \eqref{eq:cycle} with $M = [TQ/G]$.
  
  Unfortunately, $Q$ is an infinite dimensional Fr\'{e}chet manifold.  Therefore the persistence theorem does not directly apply.  However, we will be able to assert the existence of a loop which approximates the exact dynamics.  Using a phase reconstruction formula we will be able to solve for the motion of the fish.

\section{Passive Dynamics}  \label{sec:passive}
In this section we construct a Lagrangian formalism for fluid structure interaction.  Recall that given a Lagrangian, $L:TQ \to \mathbb{R}$, the equations of motion of the corresponding Lagrangian system are given in coordinates by the \emph{Euler-Lagrange equations}
\[
	\frac{D}{Dt} \left( \pder{L}{\dot{q}} \right) - \pder{L}{q} = 0.
\]
\emph{Hamilton's principle} then state that a solution curve, $q: [0 , T] \to Q$, between $q_0 = q(0)$ and $q_1 = q(T)$ satisfies the Euler-Lagrange equations if and only if $q( \cdot )$ extremizes the action
\[
	S[ q( \cdot ) ] = \int_{0}^{T}{ L( q , \dot{q}) dt }
\]
with respect to variations of the curve $q(\cdot)$ with fixed end points.  In otherwords, $\delta S=0$ with respect to variations $\delta q(t)$ with $\delta q(0) = 0$ and $\delta q (T) = 0$.  More generally, we may consider a force field, $F: TQ \to T^{\ast}Q$, and consider the Lagrange-d'Alembert equations
\[
	\frac{D}{Dt} \left( \pder{L}{\dot{q}} \right) - \pder{L}{q} = F.
\]
which are equivalent to the Lagrange-d'Alembert variational principle
\[
	\delta S = \int_0^{T}{ \lb F(q,\dot{q}) , \delta q \rb dt }
\]
with respect to variations $\delta q$ with fixed end points \cite[Chapter 7]{MandS}.
  If $Q$ is equipped with a Riemanian metric, $\lb \cdot , \cdot \rb_Q: TQ \oplus TQ \to \mathbb{R}$, then it is customary to consider Lagrangians of the form
  \[
  	L(q,\dot{q}) = \frac{1}{2}  \lb \dot{q} ,\dot{q}\rb_Q - U(q)
  \]
  where $U: Q \to \mathbb{R}$.  We call such a Lagrangian a \emph{kinetic minus potential Lagrangian}.  The Euler-Lagrange equations in this case will take the form
  \[
  	\frac{D \dot{q}}{Dt} = \nabla U(q(t))
  \]
  Where $\frac{D}{Dt}$ is the Levi-Cevita covariant derivative and $\nabla U$ is the gradient of $U$ \cite[Proposition 3.7.4]{FOM}.  Moreover, given the force $F: TQ \to T^{\ast}Q$, the Lagrange-D'Alembert equations take the form
  \[
  	\frac{D \dot{q}}{Dt} = \nabla U(q) + \sharp \left( F(q,\dot{q}) \right)
  \]
  where $\sharp : T^{\ast}Q \to TQ$ is the sharp operator induced by the metric.  In the following subsections we will illustrate how to understand a body immersed in a fluid as a kinetic-minus-potential Lagrangian system with a dissipative force field in the above sense.

\subsection{Fluids} \label{sec:passive fluids}
  Consider the manifold $\mathbb{R}^3$ with the standard flat metric and coordinates $x,y,z$.  The flat metric induces the volume form $d\vol = dx \wedge dy \wedge dz$.  One can consider the infinite dimensional Lie group of volume preserving diffeomorphisms, $\Diff_{\vol} (\mathbb{R}^3)$, where the group multiplication is simply the composition of diffeomorphisms. The configuration of a fluid flowing on $\mathbb{R}^3$ relative to some reference configuration is described by an element $\varphi \in \Diff_{\vol}(\mathbb{R}^3)$.  Given a curve $\varphi_t \in \Diff_{\vol}(\mathbb{R}^3)$ one can differentiate it to get a tangent vector $\dot{\varphi}_t = \frac{d}{dt} \varphi_t \in T \Diff_{\vol}(\mathbb{R}^3)$.  One can interpret $\dot{\varphi}$ as a map from $\mathbb{R}^3$ to $T\mathbb{R}^3$ by the natural definition $\dot{\varphi}(x) = \frac{d}{dt} \varphi_t(x)$.  Therefore, a tangent vector, $\dot{\varphi} \in T \Diff_{\vol}(\mathbb{R}^3)$, over a diffeomorphism $\varphi \in \Diff_{\vol}(\mathbb{R}^3)$ is nothing but a smooth map $\dot{\varphi}: \mathbb{R}^3 \to T\mathbb{R}^3$ such that $\tau_{\mathbb{R}^3} \circ \dot{\varphi} = \varphi$ where $\tau_{\mathbb{R}^3} : T\mathbb{R}^3 \to \mathbb{R}^3$ is the tangent bundle projection.  This implies that $\dot{\varphi} \circ \varphi^{-1}$ is a smooth divergence free vector field on $\mathbb{R}^3$.  We call $\dot{\varphi}$ the \emph{material} representation of the velocity, while $\dot{\varphi} \circ \varphi^{-1} \in \mathfrak{X}_{\vol}(\mathbb{R}^3)$ is the spatial representation (see figure \ref{fig:reps}).
  \begin{figure}[htbp] 
     \centering
     \includegraphics[width=3in]{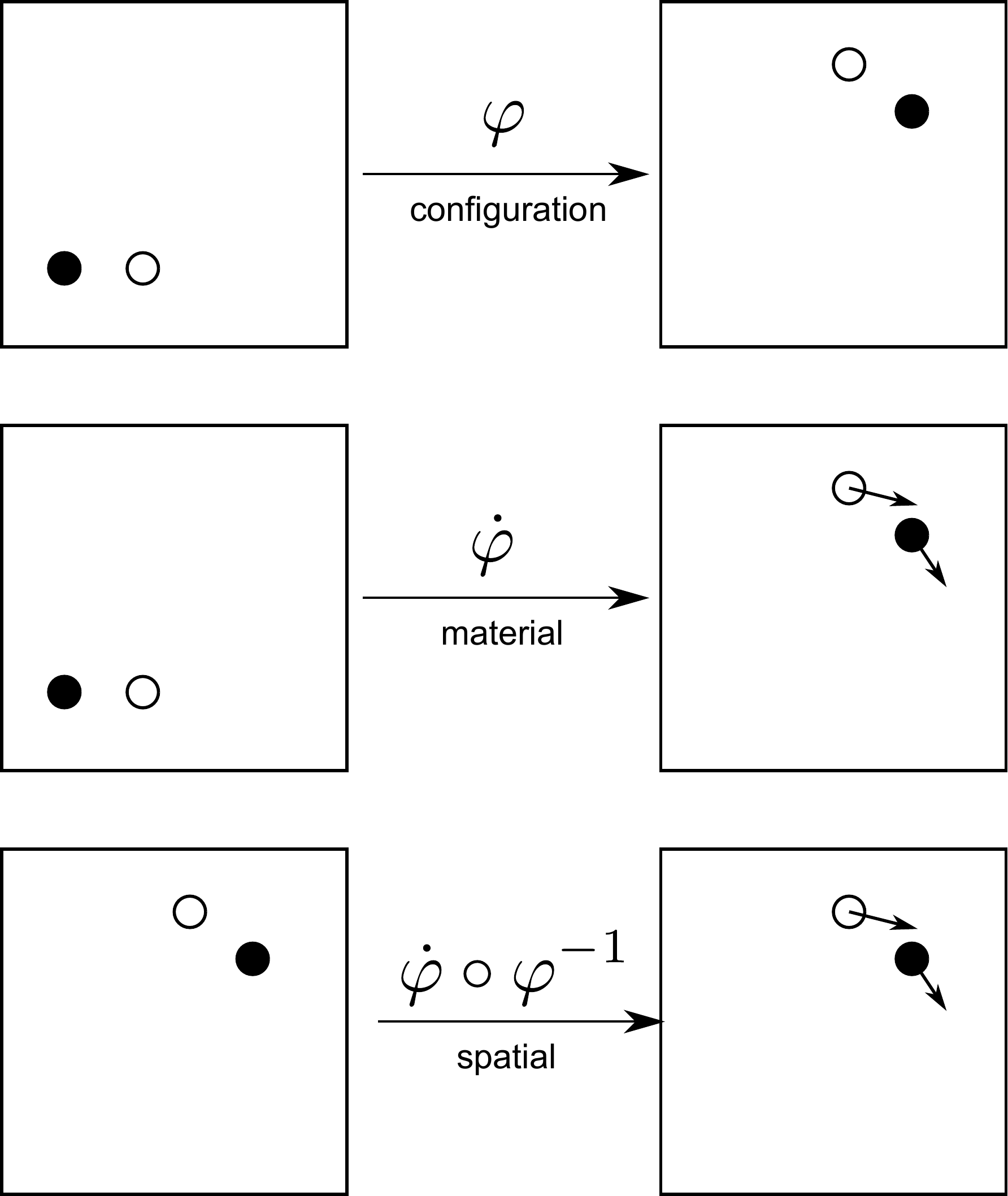} 
     \caption{A cartoon comparing the material vs spatial representation of fluids}
     \label{fig:reps}
  \end{figure}
   The Lagrangian, $L: T( \Diff_{\vol}(\mathbb{R}^3) ) \to \mathbb{R}$, is the kinetic energy of the fluid,
  \[
  	L( \varphi, \dot{\varphi}) := \frac{1}{2} \int_{M}{ \| \dot{\varphi}(x) \|^2 d\vol}.
  \]
  One can derive the Euler-Lagrange equations on $\Diff_{\vol}(\mathbb{R}^3)$ with respect to the Lagrangian $L$ to obtain the equations of motion for an ideal fluid.  However, this Lagrangian exhibits a symmetry.
  \begin{prop}[\cite{Arnold1966}]
  	The Lagrangian $L$ is symmetric with respect to the right action $\Diff_{\vol}(\mathbb{R}^3)$ on $T \Diff_{\vol}(\mathbb{R}^3)$.
\end{prop}
\begin{proof}
  Let $\dot{\varphi} \in T \Diff_{\vol}(\mathbb{R}^3)$ over the diffeomorphism $\varphi \in \Diff_{\vol}(\mathbb{R}^3)$.  Morever let $\psi \in \Diff_{\vol}(\mathbb{R}^3)$.  Then the map $\dot{\varphi} \circ \psi$ is a vector in $T \Diff_{\vol}(\mathbb{R}^3)$ over the point $\varphi \circ \psi$.  This defines a right Lie group action of $\Diff_{\vol}(\mathbb{R}^3)$ on $T \Diff_{\vol}(\mathbb{R}^3)$.  We find
  \begin{align*}
  	L(\dot{\varphi} \circ \psi) &= \frac{1}{2} \int_{M}{ \| \dot{\varphi} \circ \psi(x) \|^2 d\vol } \\
		&= \frac{1}{2} \int_{\psi^{-1}(M)}{ \| \dot{\varphi}(x)\|^2 \psi^*d \vol } \\
		&= \frac{1}{2} \int_{M}{ \| \dot{\varphi}(x) \|^2 \det( \psi) d \vol }\\
		&= \frac{1}{2} \int_{M}{ \| \dot{\varphi}(x) \|^2 d \vol} \\
		&= L( \dot{\varphi})
  \end{align*}
  which illustrates that $L$ is symmetric with respect to the natural (right) action of $\Diff_{\vol}(\mathbb{R}^3)$ on $T \Diff_{\vol}(\mathbb{R}^3)$.
\end{proof}
   Noether's theorem asserts the existence of a conserved quantity as a result of the symmetry of $L$.  This conserved quantity is the circulation of the fluid and is known as Kelvin's circulation theorem.  Additionally, this symmetry suggests we can write equations of motion on the quotient space $T \Diff_{\vol}(\mathbb{R}^3) / \Diff_{\vol}(\mathbb{R}^3)$ which one can identify with the set of divergence free vector fields, $\mathfrak{X}_{\vol}(\mathbb{R}^3)$.  It was the discover of \cite{Arnold1966} that these equations could be written as
  \[
  	\pder{u}{t} + u \cdot \nabla u = - \nabla p \quad , \quad \mathrm{div}(u) = 0.
  \]
  One should recognize as these as the inviscid fluid equations.   Moreover, if we define the linear map, $f_\mu: \mathfrak{X}_{\vol}(\mathbb{R}^3) \to \mathfrak{X}_{\vol}^*(\mathbb{R}^3)$
  \[
  	\lb f_\mu(u) , w \rb = \mu \int_{M}{ \lb \Delta u , w \rb d \vol}
  \]
  then we derive the Lagrange-D'Alemnbert equations by lifting $f$ to obtain a force field $F: T (\Diff_{\vol}(\mathbb{R}^3)) \to T^{\ast}( \Diff_{\vol}(\mathbb{R}^3))$.  If we do this, then $u(t)$ satisfies the Navier-Stokes equations
  \[
  	\pder{u}{t} + u \cdot \nabla u = - \nabla p - \mu \Delta u \quad , \quad \mathrm{div}(u) = 0.
  \]
  Of course there is no reason that we must restrict ourselves to the case of fluids on $\mathbb{R}^3$.  We can do virtually the same procedure on $\mathbb{R}^2$ or any finite dimensional Riemannian manifold for that matter.  We refer to \cite{ArKh1992} for details on these generalizations.

\subsection{Solids} \label{sec:passive solids}
  Let $\body$ be a compact manifold with boundary $\partial \body$ and volume form $d\vol_{\body}$.  Let $\Emb(\body)$ denote the set of embeddings of $\body$ into $\mathbb{R}^d$.  If $\Emb(\body)$ is non-empty then it is an infinite dimensional manifold which serves as the configuration manifold for the theory of elasticity.   For example, the theory of linear elasticity assumes $\body$ to be a Riemannian manifold with metric, $\lb \cdot , \cdot \rb_{\body}$ and uses the potential energy
  \[
  	U_{\mathrm{linear}}(b) = \frac{1}{2} \int_{b(\body)}{ \mathrm{trace}\left( [I - C_b]^T \cdot [I - C_b] \right) b_*d\vol_{\body}}.
  \]
  where $C_b$ is the right Cauchy-Green strain tensor \cite{MFOE}.  This is just an example and we will \emph{not} need our potential energy to be of this form, but this would hold the properties required for the main theorem of our paper.  For our purposes, the most important property of $U$ is (left) $\SE(d)$ invariance and an isolated minima in shape space with a nondegenerate second variation.  Let us first discuss what we mean by $\SE(d)$ invariance.  Note that the standard action of $\SE(d)$ on $\mathbb{R}^d$ induces a left action on $\Emb(B)$ by composition of maps.  Specifically, given a $z \in \SE(d)$ and a $b \in \Emb(\body)$ we define the embedding $z \cdot b \in \Emb(\body)$ by
  \[
  	(z \cdot b)(x) = z \cdot ( b(x) ) \quad , \forall x \in \body.
  \]
  To say that $U$ is $\SE(d)$ invariant means  $U( z \cdot b ) = U(b)$ for all $z \in \SE(3)$.  This implies the existence of a function $[U] : [\Emb(B)]  \to \mathbb{R}$ where $[\Emb(B)] := \SE(d) \backslash \Emb(B)$ is the quotient space, also referred to as the \emph{shape space}.  We will assume that $[U]$ has an isolated minimizing shape $s_{\min} = \operatorname{argmin}( [U] ) \in [\Emb(B)]$ and that the second variation $\delta^2 [U]$ is nondegenerate at $s_{\min}$.  Note that $s_{\min}$ is simultaneously a point in $[\Emb(B)]$ and a submanifold of $\Emb(B)$.  The Lagrangian is given by
  \[
  	L_{\body} (b,\dot{b}) := \frac{1}{2} \int_{\body}{ \rho_{\body}(x) \| \dot{b}(x) \|^2 d \vol_\body (x) } - U(b).
  \]
  From $L_{\body}$ one may compute the equations of motion for an elastic body moving in a vacuum.  Such a system would be conservative, and perhaps unrealistic.  To amend this we will include a dissipative force, which is given by a vector bundle map, $F_{\body}: T [ \Emb(\body)] \to T^{\ast} [\Emb(\body)]$, such that the tensor
\[
	\lb F_{\body}( \cdot ) , \cdot \rb : T[ \Emb(\body)] \oplus T[\Emb(\body) ] \to \mathbb{R}
\]
is a (weakly) positive definite form on $[\Emb(\body)]$.  Such a force has the effect of dampening the rate of change in the shape of the body but it will not dampen motions induced by the action of $\SE(d)$.  In other words, we assume that a jiggling body eventually comes to rest at a shape $s_{\min}$ by the dissipation of energy.

\subsection{Fluid-solid interaction} \label{sec:passive fsi}
  Let $\body$, $\Emb(\body)$, $L_{\body}$, $F_{\body}$ be as described in the previous section.  Given an embedding $b \in \Emb(\body)$ let $\bulk_{b}$ denote the set
  \[
  	\bulk_{b} = \mathrm{closure} \left\{ \mathbb{R}^d \backslash b\left( \body \right) \right\}.
  \]
  The set $\bulk_{b}$ is the region which is occupied by the fluid given the embedding of the body $b$.  If the body configuration is given by $b_0 \in \Emb(\body)$ at time $t=0$ and $b_1 \in \Emb(\body)$ at time $t=1$, then the configuration of the fluid is given by a volume preserving morphism from $\bulk_{b_0}$ to $\bulk_{b_1}$, i.e. an element of $\Diff_{\vol}\left( \bulk_{b_0}, \bulk_{b_1} \right)$.  Given a reference configuration $b_0 \in \Emb(\body)$ for the body we define the configuration manifold
  \begin{align*}
  	Q := \{ (b , \varphi) \quad \vert \quad & b \in \Emb(\body),\\
		& \varphi \in \Diff_{\vol}\left( \bulk_{b_0}, \bulk_{b} \right)\\
		 & \lim_{x \to \infty}( \varphi(x) - x ) = 0 \}.
  \end{align*}
  One should not that the manifold $Q$ has some extra structure.  In particular, the Lie group $G := \Diff_{\vol}( \bulk_{b_0} )$, represents the symmetry group for the set of particle labels and acts on $Q$ on the right by sending
  \[
  	(b,\varphi) \in Q \mapsto (b,\varphi \circ \psi) \in Q
  \]
  for each $\psi \in G$ and $(b,\varphi) \in Q$.

  \begin{figure}
    \centering
    \includegraphics[width=4in]{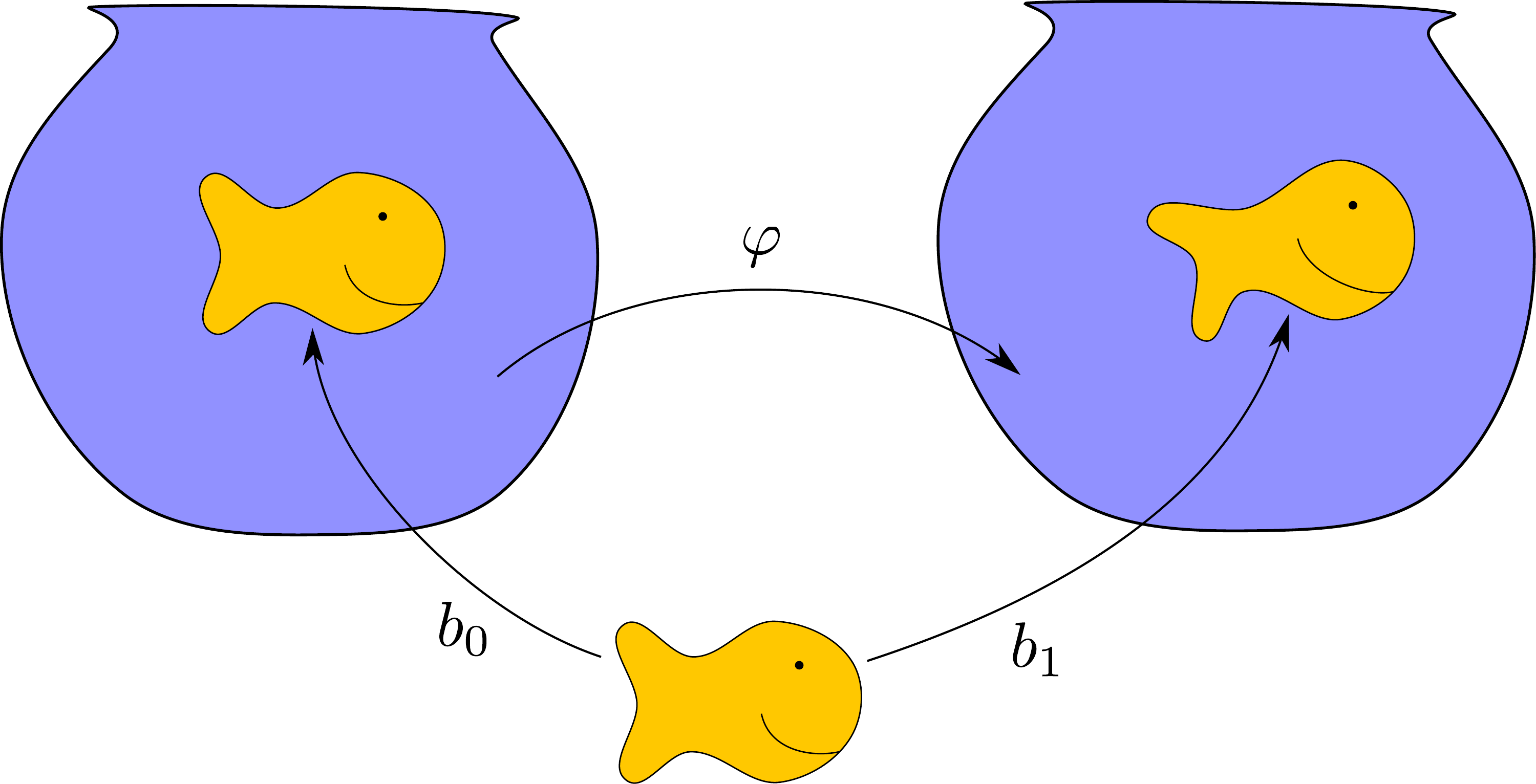}
    \end{figure}

  \begin{prop} \label{prop:Q}
  	The projection $\beta:Q \to \Emb(\body)$ defined by $\beta(b,\varphi) = b$ makes $Q$ into a principal $G$-bundle over $\Emb(\body)$.
  \end{prop}
  \begin{proof}
      The $G$-orbit of $(b,\varphi)$ is the set 
      	\[
      		\{b\} \times \Diff_{\vol}( \bulk_{b_0} , \bulk_{b} ) \subset Q.
 	\]
	  Therefore, by identifying each $G$-orbit, $\{b \} \times \Diff_{\vol}( \bulk_{b_0} , \bulk_{b} )$, with $b \in \Emb(\body)$ it is clear that the set of $G$ orbits is isomorphic to the set $\Emb(\body)$.  Thus we find that $Q / G \equiv \Emb(\body)$ and the quotient map is given by $\beta(b,\varphi) = b$.
  \end{proof}
  
  Now we must define the Lagrangian.  To do this, it is useful to note that the system should be invariant with respect to particle relablings of the fluid and so the Lagrangian should be invariant with respect to the right action of $G$ on $TQ$ given by
  \[
  	(b,\dot{b},\varphi,\dot{\varphi}) \in TQ \mapsto (b,\dot{b}, \varphi \circ \psi , \dot{\varphi} \circ \psi ) \in TQ
\]
 for each $\psi \in G$.  As a result we can define a Lagrangian on the quotient space $TQ / G$.  Incidentally, this quotient space is much closer to the space typically encountered in fluid-structure interaction because the fluid is almost represented by a vector field rather than a diffeomorphism. We say ``almost'' because the fluid velocity at the boundary is generally not tangent to the fluid domain and so this is technically not a vector field.  To be precise, the fluid is represented by a smooth map from the fluid domain to vectors above the fluid domain, i.e. an element of $\Gamma( T_{\bulk_b} \mathbb{R}^3)$.

\begin{prop} \label{prop:TQ/G}
  The quotient space $TQ / G$ is the set
 \begin{align}
 	TQ/G := \{ (b,\dot{b} , u) \quad \vert \quad & (b,\dot{b}) \in T \Emb(\body) , \nonumber \\
	& u \in \Gamma( T_{\bulk_{b}} \mathbb{R}^d ) , \nonumber \\
	& u(b(x)) = \dot{b}(x) \forall x \in \partial \body \\
	& \lim_{x \to \infty}( u(x) ) = 0 \} \label{eq:A}
 \end{align}
 equipped with the bundle projection $\tau( b,\dot{b},u) = b$ and the vector bundle structure
  \[
   (b,\dot{b}_1,u_1) + (b,\dot{b}_2, u_2) = (b, \dot{b}_1 + \dot{b}_2, u_1+u_2) 
 \]
 for all $(b,\dot{b}_1,u_1),(b,\dot{b}_2,u_2) \in \tau^{-1}(b)$ and all $ b \in \Emb(\body)$.
   Additionally, the map $\rho: (TQ/G) \to T \Emb(\body)$ given by $\rho(b,\dot{b},u) = (b,\dot{b})$ is well defined.
\end{prop}

\begin{proof}
  Let us temporarily label the set in question by ``$\mathcal{A}$'' and define the map $\pi_{/G}: TQ \to \mathcal{A}$ by \[ \pi_{/G}(b,\dot{b},\varphi,\dot{\varphi}) = (b,\dot{b}, \dot{\varphi} \circ \varphi^{-1}).\]  We see that $\pi_{/G}( v \circ \psi) = \pi_{ / G}( v )$ for all $\psi \in G$ and $v \in TQ$.  Therefore, $\pi_{/G}$ maps the coset $v \cdot G$ to a single element of $\mathcal{A}$.  Conversely, given an element $(b,\dot{b}, u) \in \mathcal{A}$ we see that $\pi_{/G}^{-1}(b,\dot{b},u)$ is the set of element in $(b,\dot{b},\varphi , \dot{\varphi}) \in TQ$ such that $u = \dot{\varphi} \circ \varphi^{-1}$.  However, this set of elements is just the coset $v \cdot G$ where $v$ is any element such that $\pi_{/G}(v) = (b,\dot{b},u)$.  Thus $\pi_{/G}$ induces an isomorphism between $TQ/G$ and $\mathcal{A}$.  Additionally we can check that $\pi_{/G}(v+w) = \pi_{/G}(v) + \pi_{/G}(w)$ and $\tau( \pi_{/G}( v ) ) = \tau_{Q}(v) \cdot G$.  Therefore, the desired vector bundle structure is inherited by $TQ/G$ as well and $\pi_{/G}$ becomes a vector bundle morphism.  Finally, the map $\rho(b,\dot{b},u) = (b,\dot{b})$ is merely the map $T\beta : TQ \to T(Q/G)$ divided by $G$.  That is to say $\rho \circ \pi_{G} = T \beta$.  This equation makes $\rho$ well defined because $\beta$ is $G$-invariant.
\end{proof}
  Just to reiterate.  The fluid velocity component $u$ in $(b,\dot{b},u)$ is not a vector-field (strictly speaking) because the boundaries of $u$ may point in directions outside of it's domain, $\bulk_{b}$.  This reflects the fact that the domain of $u$ is time dependent.

  We now define the reduced Lagrangian $l: TQ/G \to \mathbb{R}$ by
  \[
  	l(b,\dot{b},u) = L_{\body}(b,\dot{b}) + \frac{1}{2} \int_{\bulk_{b} }{ \| u(x) \|^2 d\vol(x) }.
  \]
  This induces the standard Lagrangian $L:= l \circ \pi_{/G}: TQ \to \mathbb{R}$ which is $G$-invariant by construction.  Additionally we wish to add a viscous force on the fluid, $F_{\mu} : TQ \to T^{\ast}Q$. We first define the reduced force field $f_{\mu}: TQ/G \to (TQ/G)^{\ast}$ by
  \[
  	\lb f_{\mu}( b, v_b, u) , (b, w_b , w) \rb = \mu \int_{\bulk_{b}}{ \lb \Delta u (x) , w(x) \rb d \vol(x) },
  \]
  and define $F_{\mu}: TQ(b_0) \to T^{\ast}Q(b_0)$ by
  \[
  	\lb F_{\mu}( v ) , w \rb = \lb f_{\mu}( \pi_{/G}(v) ) , \pi_{/G}(w) \rb.
  \]
    We finally define the total force on our system to be $F = F_{\mu} + (F_\body \circ T \beta)$ where $F_\body$ is the dissipative force on the shape of the body mentioned in \S \ref{sec:passive solids}.  This total force $F$ descends via $\pi_{/G}$ to a reduced force $F_{/G} : TQ/G \to (TQ/G)^{\ast}$ where $(TQ/G)^*$ is the dual vector bundle to $TQ/G$.  The reduced force is given explicitly in terms of $f_\mu$ and $F_\body$ by $F_{/G} = f_\mu + (F_\body \circ \rho)$.  One can verify directly from this expression that $\lb F( v) , w \rb = \lb F_{/G}( \pi_{/G}(v) ) , \pi_{/G}(w) \rb$.
    
    We now introduce a consequence which follows from $G$ symmetry of $F$ and $L$.
  
  \begin{prop} \label{prop:evolution1}
    Let $\Flow^t : TQ \to TQ$ denote the time-$t$ evolution operator corresponding to the Lagrange-d'Alembert equations induced by the Lagrangian $L:TQ \to \mathbb{R}$ and the force $F: TQ \to T^{\ast}Q$.  Then there exists a unique evolution operator $\flow^t : TQ/G \to TQ/G$ such that $\flow^t \circ \pi_{/G} = \pi_{/G} \circ \Flow^t$.  In other words, $\flow^t$ is the unique evolution operator which makes the diagram
\[
\begin{CD}
TQ @> \Flow^t >>  TQ\\
@VV{\pi_{/G}}V @VV{\pi_{/G}}V\\
TQ/G @> \flow^t >> TQ/G
\end{CD}
\]
    commute.
  \end{prop}
  
  \begin{proof}
    Let $q:[0,t] \to Q$ be a curve such that the time derivative $(q,\dot{q}): [0,t] \to TQ$ is an integral curve of the Lagrange-d'Alembert equations with initial condition $(q,\dot{q})(0)$ and final condition $(q,\dot{q})(t)$.  Then the Lagrange-d'Alembert variational principle states that
    \[
    	\delta \int_{0}^{t}{ L( (q,\dot{q})(\tau)) d\tau } = \int_{0}^{t}{ \lb F((q,\dot{q})(\tau) , \delta q(\tau) \rb d\tau}
    \]
    for all variations of the curve $q( \cdot)$ with fixed endpoints.  Note that for each $\psi \in G$ the action satisfies $\int_{0}^{t}{ L((q,\dot{q})(\tau) d\tau} = \int_{0}^{t}{ L( (q,\dot{q})(\tau) \circ \psi ) d\tau}$ and the variation on the right hand side of the Lagrange-d'Alembert principle is
   \begin{align*}
    	\int_{0}^{t}{ \lb F((q,\dot{q})(\tau)) , \delta q (\tau) \rb d\tau} &= \int_{0}^{t}{ \lb F_{/G}( \pi_{/G}( (q,\dot{q})(\tau) , \pi_{/G}( \delta q (\tau)) \rb d\tau} \\
	&= \int_0^t{ \lb F( (q,\dot{q})(\tau) \circ \psi ) , \delta q(\tau) \circ \psi \rb d\tau}.
   \end{align*}
   Therefore we observe that
   \[
    	\delta \int_{0}^{t}{ L( (q,\dot{q})(\tau) \circ \psi) d\tau } = \int_{0}^{t}{ \lb F((q,\dot{q}) (\tau) \circ \psi) , \delta q (\tau) \circ \psi \rb d\tau}
   \]
   for arbitrary variations of the curve $q( \cdot )$ with fixed end points.  However, the variation $\delta q \circ \psi$ is merely a variation of the curve $q \circ \psi (\cdot)$ becuase
   \[
   	\delta q(\tau) \circ \psi = \left. \pder{}{\epsilon} \right|_{\epsilon = 0}( q(\tau, \epsilon) \circ \psi )
   \]
   and if $q(\tau,\epsilon)$ is a deformation of $q( \tau)$ then $q(\tau,\epsilon) \circ \psi$ is a deformation of $q(\tau) \circ \psi$ by construction.
     Therefore,
   \[
    	\delta \int_{0}^{t}{ L( (q,\dot{q})(\tau) \circ \psi) d\tau} = \int_{0}^{t}{ \lb F((q,\dot{q})(\tau) \circ \psi) , \delta (q \circ \psi) \rb d\tau}
   \]
   for arbitrary variations of the curve $q \circ \psi$ with fixed end points.  This last equation states that the curve $(q,\dot{q}) \circ \psi$ satisfies the Lagrange-d'Alembert principle.  Thus $\Flow^t((q,\dot{q})(0) \circ \psi)=(q,\dot{q})(t) \circ \psi$.  Since $\psi \in G$ was chosen arbitrarily we may apply $\Flow^t$ to the entire coset $(q,\dot{q}) \cdot G$ to find 
   \begin{align}
   	\Flow^t( (q,\dot{q} ) \cdot G) = \Flow^t(q,\dot{q}) \cdot G. \label{eq:Flow_to_flow}
  \end{align}
  This last equation describes a map from cosets to cosets,  that is from $TQ/G$ to $TQ/G$.  In other words, equation \eqref{eq:Flow_to_flow} is the map $\flow^t : TQ/G \to TQ/G$.
  \end{proof}

\subsection{Reduction by frame invariance} \label{sec:SE}
  Consider the group $\SE(d)$ consisting of rotations and translations of $\mathbb{R}^d$.  Each $z \in \SE(d)$ sends $(b,\dot{b}, u) \in TQ/G$ to $(z (b,\dot{b}) , z_* u ) \in TQ/G$ where $z_*u$ is the push-forward of the fluid velocity field to the domain $\bulk( z \cdot b)$.  This action is free and proper on $TQ/G$ so that the projection $[ \cdot ] : TQ/G \to [ TQ / G]$ where
 \[
 	[TQ/G] := \frac{TQ/G}{\SE(d)}
 \]
 is a principal bundle \cite[Prop 4.1.23]{FOM}. Additionally, $\SE(d)$ acts by vector bundle morphisms which are isomorphisms on each fiber so that $[TQ/G]$ inherits the vector-bundle structure from $TQ/G$.
   \begin{prop}
     There exists a function $[l] : [TQ/G] \to \mathbb{R}$ such that $l \circ [\cdot ] = [l]$, a vector-bundle projection $[\tau]: [TQ/G] \to [\Emb(\body)]$, a map $[\rho] : [TQ/G] \to T [\Emb(\body) ]$ such that the diagrams
     \begin{align*}
     \begin{CD}
TQ/G @> [\cdot] >>  [TQ/G] \\
@VV{\tau}V @VV{[\tau]}V\\
\Emb(\body) @> [\cdot] >> [\Emb(\body)]
\end{CD} \quad , \quad
 \begin{CD}
TQ/G @> [\cdot] >>  [TQ/G] \\
@VV{\rho}V @VV{[\rho]}V\\
T\Emb(\body) @> T[\cdot] >> T[\Emb(\body)]
\end{CD}
     \end{align*}
     commute.  Lastly, there exists a force $[F_{/G}] : [TQ/G] \to [TQ/G]^{\ast}$ define by the equation $\lb [F_{/G}]( [ \xi ] ) , [\eta] \rb = \lb F_{/G}(\xi) , \eta \rb$ for any $\xi, \eta \in TQ/G$ above the same base point. 
   \end{prop}
   \begin{proof}
     We first show the existence of $[l]$.  Let $z \in \SE(d)$ and $\xi = (b,\dot{b},u) \in TQ/G$.  Note that when $z$ acts on a vector $(x,\delta x) \in T \mathbb{R}^d$ it will preserve the length of the the vector.  That is to say $\| \delta x \|^2 = \| z \cdot \delta x\|^2$.  Thus we see that 
     \begin{align*}
     	l( z  \cdot \xi ) &= L_{\body}(z \cdot (b,\dot{b}) ) + \int_{\bulk(z \cdot b)}{ \| z_*u(x) \|^2 d\vol} \\
		&= L_{\body}(b,\dot{b}) + \int_{\bulk_{b}}{ \| u(x) \|^2 d \vol } \\
		&= l(\xi)
     \end{align*}
     where we have used the change of variables formula and the fact that $L_{\body}$ is $\SE(d)$ invariant by assumption.  Since $z$ is aribtrary we find that applying $l$ to the entire coset $\SE(d) \cdot \xi$ yields a single real number.  This means we may define a function $[l]: [TQ/G] \to \mathbb{R}$ by the condition $l = [l] \circ [\cdot]$.  Continuing in this manner we see that $\tau( z \cdot \xi ) = z \cdot \tau( \xi)$ so that there is a well defined map $[\tau]: [TQ/G] \to [\Emb(\body)]$ defined by the condition $\tau = [\tau] \circ [\cdot]$ and the analogous argument map be applied to the map $\rho: TQ/G \to T\Emb(\body)$ to yield a map $[\rho]: [TQ/G] \to T[\Emb(\body) ]$.  Lastly, we note that
     \begin{align*}
     	\lb F_{/G}( z \cdot \xi) , (z \cdot \eta) \rb &= \lb F_{\body}( z \cdot (b,\dot{b}) , z \cdot (b,\dot{b}' \rb + \int_{\bulk{ z b }}{ \lb \Delta ( z_* u)(x) , z_* u'(x) \rb d\vol} \\
		&= \lb F_{\body}(b,\dot{b}) , (b,\dot{b}') \rb + \int_{\bulk_{b}}{ \lb \Delta u(x) , u'(x) \rb d \vol} \\
		&= \lb F_{/G}(\xi) , \eta \rb.
     \end{align*}
    Where $\xi = (b,\dot{b},u)$ and $\eta = (b,\dot{b}', u')$ are arbitrary element of $TQ/G$ over the same fiber and $z $ is an arbitrary element of $\SE(d)$.  Since $z$ is arbitrary, there must exist a unique function $[F_{/G}] : [TQ/G] \to [TQ/G]^{\ast}$ such that $\lb [F_{/G}]( [\xi]) , [\eta] \rb = \lb F_{/G}(\xi) , \eta \rb$ for arbitrary $\xi,\eta \in TQ/G$.
   \end{proof}
   In the preceding proof we also illustrated a few other things.  In particular, the Lagrangian $L$ and the unreduced force $F$ are also $\SE(d)$ invariant under the left action $z \cdot (b,\varphi) = (z \cdot b , z_* \varphi)$ for each $z \in \SE(d)$ and $(b,\varphi) \in Q$.  Therefore the flow of the system is also invariant with respect to the action of $\SE(d)$ and so we may define a flow on a space which has been quotiented a second time by this left action.
   
   \begin{prop} \label{prop:evolution2}
     Let $\flow^t : TQ/G \to TQ/G$ be the time $t$ evolution operator of Proposition \ref{prop:evolution1}.  There exists an evolution operator $[\flow]^t : [TQ/G] \to [TQ/G]$ such that the following diagram
     \[
     	\begin{CD}
		TQ/G @> \flow^T >> TQ/G \\
		@VV{[\cdot]}V @VV{[\cdot]}V \\
		[TQ/G] @> [\flow]^T >> [TQ/G]
	\end{CD}
     \]
     commutes.
   \end{prop}
  \begin{proof}
    Let $q:[0,T] \to Q$ be a curve such that the time derivative $(q,\dot{q}): [0,T] \to TQ$ is an integral curve of the Lagrange-d'Alembert equations with initial condition $(q,\dot{q})_0$ and final condition $(q,\dot{q})_T$.  Then the Lagrange-d'Alembert variational principle states that
    \[
    	\delta \int_{0}^{T}{ L( q,\dot{q}) dt } = \int_{0}^{T}{ \lb F(q,\dot{q}) , \delta q \rb dt}
    \]
    for all variations of the curve $q( \cdot)$ with fixed endpoints.  Note that for each $\psi \in G$ and $z \in \SE(d)$ the action satisfies $\int_{0}^{T}{ L(q,\dot{q}) dt} = \int_{0}^{T}{ L( z \cdot (q,\dot{q}) \circ \psi ) dt}$ and the variation in the work is
   \begin{align*}
    	\int_{0}^{T}{ \lb F(q,\dot{q}) , \delta q \rb dt} &= \int_{0}^{T}{ \lb F_{/G}(z \cdot  \pi_{/G}( q,\dot{q}) , z \cdot \pi_{/G}( \delta q) \rb dt} \\
	&= \int_0^T{ \lb F( z \cdot (q,\dot{q}) \circ \psi ) , z \cdot \delta q \circ \psi \rb dt}.
   \end{align*}
   Therefore we observe that
   \[
    	\delta \int_{0}^{T}{ L( z \cdot (q,\dot{q}) \circ \psi) dt } = \int_{0}^{T}{ \lb F( z \cdot (q,\dot{q}) \circ \psi) , z \cdot \delta q \circ \psi \rb dt}
   \]
   for arbitrary variations of the curve $q( \cdot )$ with fixed end points.  However, the variation $z \cdot \delta q \circ \psi$ is merely a variation of the curve $z \cdot q \circ \psi (\cdot)$.  Therefore,
   \[
    	\delta \int_{0}^{T}{ L( z \cdot (q,\dot{q}) \circ \psi) dt } = \int_{0}^{T}{ \lb F( z \cdot (q,\dot{q}) \circ \psi) , \delta (z \cdot q \circ \psi) \rb dt}
   \]
   for arbitrary variations of the curve $z \cdot q \circ \psi$ with fixed end points.  This last equation states that the curve $z \cdot (q,\dot{q}) \circ \psi$ satisfies the Lagrange-d'Alembert principle.  Thus $\Flow^T(z \cdot (q,\dot{q})_0 \circ \psi)= z \cdot (q,\dot{q})_T \circ \psi$.  Since $\psi \in G$ was arbitrary we may apply $\Flow^T$ to the entire coset $z \cdot (q,\dot{q}) \cdot G$ to find 
   \[
   	\Flow^T( z \cdot (q,\dot{q} ) \cdot G) = z \cdot \Flow^T(q,\dot{q}) \cdot G.
  \]
  This last equation states that
  \[
  	\flow^T(z \cdot \xi) = z \cdot \flow^T(\xi)
  \]
  where $\xi = \pi_{/G}((q,\dot{q})_0)$.  However, as $z \in SE(d)$ is arbitrary we can multiply by all of $SE(d)$ to get
   \begin{align}
   	\flow^t( \SE(d) \cdot \xi) = \SE(d) \cdot \flow^t(\xi). \label{eq:flow_to_flow}
  \end{align}
  This last equation describes a map from cosets to cosets,  that is from $[TQ/G]$ to $[TQ/G]$.  In other words, equation \eqref{eq:flow_to_flow} is the map $[\flow]^t : [TQ/G] \to [TQ/G]$.
\end{proof}

\section{Asymptotic Behavior} \label{sec:asy}
  Our experience of observing a dumpling floating in a bowl of soup should suggests that the passive motion of the system tends towards one where the fluid is stagnant and the shape of the solids has settled to some minima of the elastic potential energy.  In this section we will confirm this using language presented thus far.

\begin{prop} \label{prop:stable_set}
Let $q:[0,\infty) \to Q$ be a curve such that the time derivative $(q,\dot{q}) : [0, \infty) \to TQ$ is an integral curve of the Lagrange-d'Alembert equations for the Lagrangian $L$ and the force $F$.  Then the $\omega$-limit set of $(q,\dot{q})( \cdot )$ is contained in the set
    \[
      dU^{-1}(0) := \{ (q,0) \in TQ \quad \vert \quad dU(q) = 0 \}.
    \]
\end{prop}
\begin{proof}
  Roughly speaking the proof goes as follows.  We find that the energy is a quadratic positive definite funtion.  The time derivative of the energy is negative as along as the velocity is non-zero.  This is due to dissapation from viscous frictions.  The positive-definitness of the energy combined with its negative time derivative will allow us to conclude that the velocity is sent to zero.  That is to say, the system evolves towards the zero section of $TQ$.  We will then find that accelleration on the zero-section vanishes if and only if $dU$.  This means that the system goes to points in the zero section where $dU$ vanishes and concludes the proof.

  We will now proceed to prove this in a more explicit manner. Before we begin we will introduce some handy notation. Given a vector bundle $\pi:V \to M$ there exists a unique section called the \emph{zero-section} which maps each $m \in M$ to the $0$-vector in $V$ above the point $m$.  We will denote this section for an arbitrary vector bundle by $( \cdot )_{\uparrow}^0$.  This point is minor perhaps, becase it is customary to identify the zero-section of a vector bundle with the base of the bundle.  However, to avoid causing confusion we will not make this identification.

  The energy is the function $E: TQ \to \mathbb{R}$ given by
  \begin{align*}
  	E(q,\dot{q}) &:= \lb \mathbb{F}L(q,\dot{q}) , \dot{q} \rb - L(q,\dot{q}) \\
		&= \frac{1}{2} \int_{\body}{ \rho_{\body}(x) \| \dot{b}(x) \|^2 d\vol_{\body}} + \frac{1}{2} \int_{\bulk_{b_0}}{ \| \dot{\varphi}(x) \|^2 d \vol} - U(b)
  \end{align*}
   Given any Lagrangian system on a Riemannian manifold where the Lagrangian is the kinetic energy minus the potential energy, the time derivative of the generalized energy under the evolution of the Lagrange-d'Alembert equations is given by $\dot{E} = \lb  F(\dot{q}) , \dot{q} \rb$.  In this case we find
   \[
   	\dot{E}(q,\dot{q}) = \lb F_{\body}(\dot{b}) , \dot{b} \rb + \lb F_{\mu}( q,\dot{q}), (q,\dot{q}) \rb.
   \]
   However, the right hand side is a (weakly) positive definite quadratic form on each fiber of $TQ$.  Therefore the $\omega$-limit of $(q,\dot{q})( \cdot )$, denoted $M^{\omega}$, must be a subset of the zero section of $TQ$, which is identifiable with $Q$ itself.  In other-words, $M^{\omega} \subset Q$.  Let $(b,\varphi) = q \in M^{\omega}(q,\dot{q})$.  Then the Lagrange-D'Alembert equations is equivalent to
   \[
   	\frac{D \dot{q}}{Dt} = \nabla U(b) + \sharp( F(q,\dot{q}) )
   \]
   where $\sharp:T^{\ast}Q \to TQ$ is the sharp map associated the metric on $Q$.  However, $F(q,\dot{q}) = 0$ when $\dot{q} = 0$, which is the case for point in $M^{\omega}(q,\dot{q})$.  Thus, the evolution of the system is given by the eqaution
   \[
   	\frac{D \dot{q}}{Dt} = \nabla U(b)
   \]
   for points in the zero section of $TQ$.  The covariant derivative takes place above a vector with zero velocity so that $\ddot{q} \in TTQ$ is a vertical vector.  The above equation says that the vertical part of $\ddot{q}$ is given by $\nabla U$.  There the vector field on $TQ$ is transverse to the zero section of when $dU \neq 0$.  Hence we find that we can restrict $M^{\omega}$ further.  That is to say, $M^{\omega} \subset dU^{-1}(0)$.
\end{proof}

\begin{corr} \label{corr:stable_point}
  Let $[U]: [\Emb(\body)] \to \mathbb{R}$ be the unique function on the shape-space of the body such that $[U]( [ b] ) = U(b)$ for all $b \in \Emb(\body)$.  Assume that $[U]$ has a unique minimizer $s_{\min} \in [\Emb(\body)]$.  Then if $(q,\dot{q}): [0,\infty) \to TQ$ is an integral curve of the Lagrange-d'Alembert equations, then $[\xi] (t) = [ \pi_{/G}(q,\dot{q}(t))]$ must approach $(s_{\min})^{0}_{\uparrow} \in [TQ/G]$.  If the flow of the Lagrange-d'Alembert equations is complete, this means that $(s_{\min})^0_{\uparrow}$ is a global (weakly) hyperbolically stable fixed point for the vector field $X_{[TQ/G]}$.
\end{corr}
\begin{proof}
  In proposition \ref{prop:stable_set} we showed that solutions approach points within the set $dU^{-1}(0)$ asymptotically.  This implies that the dynamics on $[TQ/G]$ must approach $d[U]^{-1}(0)$ asymptotically.  However, there is only one such point.
\end{proof}

We can interpret corrollary \ref{corr:stable_point} to mean that ``a motionless body in stagnant water is a stable state for dead fish".  In the next section we will periodically perturb this stable equilibria to obtain a loop in $[TQ/G]$.

\section{Swimming} \label{sec:swimming}
  When one activates one's muscles they do so by changing chemical potentials which then stiffen muscle fibers, causing them to contract.  In terms of Lagrangian mechanics, this has the effect of altering the elastic potential energy of the muscle tissue.  Therefore to model the periodic muscle activation of a fish swimming we may consider a time periodic potential energy $\tilde{U}: S^1 \times [\Emb(\body)] \to \mathbb{R}$, which we define on shape space.  If we add this potential energy to our Lagrangian we do not break any of the symmetries discussed thus far.  However, this does alter the dynamics.  In particular, adding $\tilde{U}$ to the Lagrangian is equivalent to adding a time periodic force $\tilde{F} = d\tilde{U}$ to the equations of motion.  If we let $X_{TQ} \in \mathfrak{X}(TQ)$ denote the vector field for the flow, $\Flow^t$, then the flow in the time-periodic augmented phase space is given by the vector field $(\partial_t , X_{TQ}) \in \mathfrak{X}(S^1 \times TQ)$.  Adding a periodic potential energy effects the dynamics in augmented phase space by the addition of a vector-field $Y_{TQ}: S^1 \times TQ \to TTQ$ \cite[\S 7.8]{MandS}.  In particular, the dynamics of the periodically perturbed system in augmented phase space are given by $(\partial_t , X_{TQ} + Y_{TQ} ) \in \mathfrak{X}(S^1 \times TQ)$.  If $\tilde{U}$ only depends on the shape of the body, then $Y_{TQ}$ is both $G$ and $\SE(d)$ invariant and thre exists a vector field $Y_{[TQ/G]} \in \mathfrak{X}([TQ/G])$ so that the dynamics on $[TQ/G]$ in augmented phase space are given by $(\partial_t , Y_{[TQ/G]} + X_{[TQ/G]}) \in \mathfrak{X}(S^1 \times [TQ/G])$.  Because $X_{[TQ/G]}$ exhibits an asymptotically stable fixed point , $x_0 \in [TQ/G]$, then the orbit $\{ (x_0, t) : t \in S^1 \} \subset S^1 \times [TQ/G]$ must be a (weak) hyperbolic stable limit cycle for the vector field $(\partial_t, X_{[TQ/G]})$.  If we were observing a finite dimensional manifold we could invoke the persistence theorem to state the existence of a perturbed limit cycle in the dynamics of $(\partial_t, X_{[TQ/G]} + Y_{[TQ/G]})$ for sufficiently small oscillations.  However, this final leap is problematic.
  
\subsection{Analytical Issues and a work-around}
    Unfortunately, our manifold is infinite dimensional and Fr\'echet.  I am unable to assert anything with absolute certainty on the existence of limit cycles in $[TQ/G]$ for arbitrary perturbations.  Specifically, the strength of the stable fixed point of $x_0$ is tied to the spectrum of the Laplacian operator on $\bulk_b$, which is strictly negative but approaches $0$.  In the language of \cite{Fenichel1971}, there is no spectral gap condition.  Secondly, even if there were a spectral gap, the existence of a limit cycle involves fixed point theorems which assume that our space is complete.

    However, perhaps these problems are not so devastating. There exists a number of finite dimensional models for the space $TQ/G$ used by engineers to study fluid structure interaction.  It is fairly common to approximate the fluid velocity field on a finite dimensional space and model the solid using a finite element method.  For example consider the immersed boundary method \cite{Peskin2002}.  Let us call this space $(TQ/G)_{\discrete}$.  Moreover, usually one can act on $(TQ/G)_{\discrete}$ by $\SE(d)$ simply by rotating and translating the tetrahedra in the case of finite element methods, or rotating the basis functions for a spectral method.  If the model on $(TQ/G)_{\discrete}$ converges as the time step goes to zero then we could reasonably restrict ourselves to methods  which dissipate energy at a rate which is quadratic and positive definite in velocity.  This is not too much to expect because a good method ought to converge and \cite{Peskin2002} or \cite{Monaghan1977} are both methods which exhibit this property.  By the same arguments as before the dynamics will exhibit a hyperbolically stable equilibria on the quotient space $\frac{ (TQ/G)_{\discrete} }{\SE(d)}$.  Upon adding a periodic perturbation to the dynamics on $\frac{ (TQ/G)_{\discrete} }{\SE(d)}$ one could apply the persistence theorem directly to assert the existence of a hyperbolically stable limit cycle in augmented phase space.  If the model on $\frac{ (TQ/G)_{\discrete} }{\SE(d) }$ converges as the grid size goes to zero, then for a finite grid resolution there is a trajectory in $TQ/G$ which is well approximated by the definition of the term ``convergence''.  Therefore one could say that this limit cycle approximates the real dynamics, although perhaps only on a finite time span.  That is to say, there is a loop in $[TQ/G]$ which nearly satisfies the exact dynamics.  Therefore instead of searching for conditions under which a limit cycle exists for the exact dynamics we shall explore the consequence of evolving on loops in $[TQ/G]$ which \emph{approximate} the exact dynamics.  Given one of these loops we will be able to interpret regular motion (a.k.a swimming) by the use of phase reconstruction formulas induced by lifting the loop in $[TQ/G]$ to a path in $TQ/G$.
 
\subsection{Swimming as a limit cycle}

\begin{figure}[h!]
  \centering
  \includegraphics[width=4in]{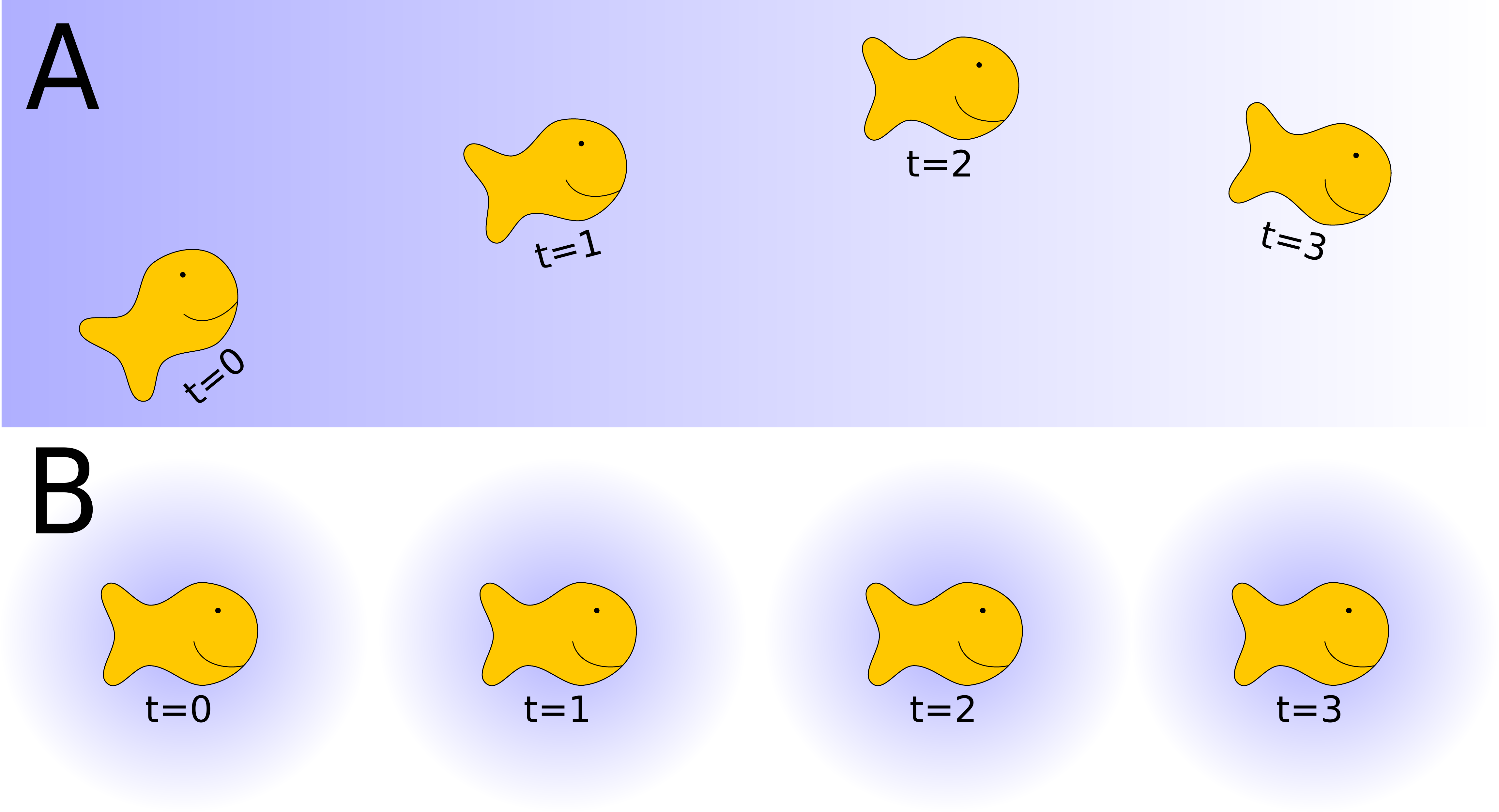}
  \caption{The bottom row (B) depicts snapshots of a periodic trajectory in $[TQ/G]$.  As the trajectory is periodic, all the snapshots are identical.   The top row (A) is a lift of the trajectory in (B) to $TQ/G$.  The top row appears to be steady swimmming.  Each snapshot is related to the previous by a (fixed) rigid rotation and translation.}
  \label{fig:snapshots}
\end{figure}

Let  $[\gamma]: S^1 \to [TQ/G]$ be a loop such that $\graph([\gamma]) := \{ (t, [\gamma](t) ) : t \in S^1\}$ is approximates the dynamics in the augmented phase space $S^1 \times [TQ/G]$.  There must exists a curve $\gamma^{(1)} : [0 , 1] \to TQ/G$ such that $[\gamma^{(1)}(t)] = [\gamma](t)$ for $t \in [0,1]$.  Additionally, since $[\gamma](0) = [\gamma](1)$ there must exists a $z \in \SE(d)$ such that $\gamma^{(1)}(1) = z \cdot \gamma^{(1)}(0)$.  However, we see that $z \cdot \gamma^{(1)}(t)$ is a smooth curve which projects to the loop $[\gamma]$.  Therefore, by concatenating the curves $\gamma^{(1)}( \cdot )$ and $z \cdot \gamma^{(1)}( \cdot)$ we can get a longer integral curve $\gamma^{(2)} : [0,2] \to TQ/G$ defined by
\[
	\gamma^{(2)}(t) = \begin{cases}
		\gamma^{(1)}(t) &\text{ for } t \in [0,1] \\
		z \cdot \gamma^{(1)}(t-1) &\text{ for } t \in (1,2]
		\end{cases}.
\]
Note that $\gamma^{(2)}(2) = z \cdot \gamma^{(1)}(1) = z^2 \cdot \gamma^{(1)}(0)$.  Observing that $z \gamma^{(2)}(\cdot)$ projects to $[\gamma]$ as well we can extend $\gamma^{(2)}( \cdot)$ to get a curve $\gamma^{(3)} : [0 , 3] \to TQ/G$.  By induction we can extend this argument indefinetly to get a curve for all positive time
\[
	\gamma(t) = z^{ \lfloor t \rfloor } \cdot \gamma^{(1)}(t - \lfloor t \rfloor)
\]
where $\lfloor t \rfloor = \sup \{ k \in \mathbb{Z} : k \leq t \}$.   One way to interpret this visually is as follows. If one were to take a snap shot of the body immersed in the fluid every second, then one would observe the same picture over and over again, except each picture would be translated and rotate by the fixed element $z \in \SE(d)$ (see figure \ref{fig:snapshots}).  Moreover, since $[\gamma]( \cdot )$ is a  \emph{stable} limit cycle, this behavior is stable and so we should expect to observe it in real systems (see figure \ref{fig:jellyfish}).  Of course, it is possible that the reconstructed phase-shift $z \in \SE(d)$ is simply the identity, which would mean there has been no net motion after one flap of the fins.  The main point is that this regular behavior is stable and generally $z$ will not be the identity.  Constructing ways of controlling the phase-shift $z$ will certainly bring us into future work, which is the topic of the next section. 

\section{Conclusion and Future work}
  It is widely observed that steady swimming is periodic, and this observation inspired the question ``is it possible to interpret swimming as a limit cycle?''.  In this paper we have answered this question in the affirmative.  We did this by carefully constructing the phase space and then reducing it by the appropriate symmetries.  We arrived at a system where we could invoke the persistence theorem to construct loops that approximate the dynamics in the reduced phase space.  Lifting these loops by phase reconstruction formulas allowed us to assert the existence of stable trajectories which strongly resemble swimming.  The stability and the regularity of these trajectories were important features, as both of these features are observed in real systems.
  
    Given the complexity of fluid-structure interaction it is not immediately clear that one should have expected such orderly behavior.  This orderliness has the potential to be exploited in a number of applications.

\begin{enumerate}
\item {\bf Robotics and Optimal Control}
  The interpreation of swimming as a limit cycles may permit a new paradigm for controller design.  For example, if we assume that we have a potential energy which is a sum of two parts $U([b]) + U_{c}([b])$ where the $U_c$ term depends on a control parameter $c \in C$. We can denote the set of loops in $C$ by $\mathrm{loop}(C)$ then we may consider the map $\Gamma: \mathrm{loop}(C) \to \mathrm{loop}( [TQ/G] )$ which outputs the periodic limit cycle in $[TQ/G]$ which results from using a loop in $C$.  We may use $\Gamma$ to define our cost function on the space $\mathrm{loop}(C)$.  Such a cost functional would be one which is sensitive to the long term behavior and does not over-react to the transient dynamics.
 
\item {\bf Transient dynamics} Even though trajectories may approach a limit cycle, the transient dynamics will matter.  The transient dynamics will re-orient and translate the body before any orderly periodic behavior kicks in.  Therefore, if one desires to create locomotion through periodic control inputs, one should try to get onto a limit cycle quickly in order to minimize the duration where the complicated transient dynamics dominate.  Both these areas can be studied if and only if the map $\Gamma$ exists.  A main contribution of this paper a demonstration of an ``approximate'' map $\Gamma$.

\item {\bf Pumping}
  In the current setup one could consider a reference frame attached to the body.  In this reference frame ``swimming'' manifests as fluid moving around the body in a regular fashion.  This change in our frame in reference is really a description of pumping and suggests the stability of fluid pump may be proven in a similar manner.

\item {\bf Passive Dynamics}
  This paper does \emph{not} address the dual problem.  By the dual problem we mean: ``Given a constant fluid velocity at infinity, what periodic motion (if any) will a tethered body approach as time goes to infinity?''  In this dual problem, the motion of the body is given first, and parameters such as the period of the limit cycle are emergent phenomena.  In particular, the dual problem of a flapping flag immersed in a fluid with a constant velocity at infinity has received much attention in the applied mathematics community (see \cite{Shelley2005} and references therein).  It is aboserved that stable periodic motion occurs for a range of velocities at infinity while resonant modes of the flag lead to chaotic behavior in certain regions.  This problem has been studied from a variety of angles already.  However, it may be insightful to view it in a manner similar to the analysis presented in this paper.  Such an interpretation could help generalize well known results which would otherwise be specific to flapping flags.

\item {\bf Other types of locomotion}
  The notion that walking may be viewed as a limit cycle has been around for a while (see \cite{Hobbelen} and references therein).  Moreover, it is conceivable that flapping flight is a limit cycle as well.  However, for both these systems $\SE(3)$ symmetry is broken by the direction of gravity.  Because of this, it is not immediately clear that one can import the methods used here to understand flapping flight and terrestrial locomotion.  However, perhaps this is merely a challenge to be overcome.  In particular, these systems still exhibit $\SE(2)$ symmetry.  So it is conceivable that there exists a fixed point for a reduced system which is hyperbolic but with an unstable mode.  Using the same construction one can infer the existence of a periodic orbit which is not a limit cycle.  It would be interesting to see if we could describe the unstable directions physically and inspire controls based upon these insights.
\end{enumerate}
\bibliographystyle{amsalpha}

\bibliography{hoj_swimming}

\end{document}